\documentclass[12pt]{article}
\usepackage[russian]{babel}
\usepackage[cp1251]{inputenc}
\usepackage{amsmath,latexsym}
\usepackage[psamsfonts]{amssymb}
\usepackage{amssymb,amsmath,amsthm,amsfonts}
\usepackage{indentfirst}
\newtheorem{theorem}{Теорема}
\newtheorem{lemma}{Лемма}
\newtheorem{corollary}{Следствие}

\begin{document}

УДК 517. 955
\begin{center}
{\Large\bf АНАЛОГ ЗАДАЧИ КОШИ ДЛЯ НЕОДНОРОДНОГО МНОГОМЕРНОГО ПОЛИКАЛОРИЧЕСКОГО УРАВНЕНИЯ С ОПЕРАТОРОМ БЕССЕЛЯ}

\medskip
{\large
 Ш.Т.Каримов}
\end{center}
{ \it
Каримов Шахобиддин Туйчибоевич ---
кандидат физико-математических наук, доцент,
заведующий кафедрой математики,
Ферганский государственный университет.
г.Фергана, ул. Мураббийлар 19, ФерГУ, 150100, Республика Узбекистан. \,\,
E-mail: shkarimov09@rambler.ru\\}

\begin{abstract}
В работе получена явная формула решения аналога задачи Коши для неоднородного многомерного поликалорического уравнения с оператором Бесселя.  Для построения решения применен многомерный оператор Эрдейи - Кобера дробного порядка. Библ.: 21 назв.

\medskip

{\it Ключевые слова:} Задача Коши, поликалорическое уравнение, оператор Эрдейи-Кобера, дифференциальный оператор Бесселя.
\end{abstract}

\medskip

\section{Введение. Постановка задачи.}

Сингулярные параболические уравнения с оператором Бесселя относятся к классу уравнений,  вырождающихся по пространственным переменным на границе области, и они часто встречаются в приложениях. Например, в задачах теплопереноса в неподвижной среде (твердом теле), в задачах диффузионного пограничного слоя \cite{polyanin}, в задачах распространения тепла при закачке горячей жидкости в нефтяной пласт \cite{chuprov} и в других разделах науки и техники.

Вырождающиеся уравнения представляют собой один из важных разделов современной теории дифференциальных уравнений с частными производными и число опубликованных работ по этой тематике весьма значительно. В этом направлении особое место занимают начальные и краевые задачи для параболических уравнений с оператором Бесселя. Теория классических решений задачи Коши для сингулярных параболических уравнений второго порядка с оператором Бесселя развита в работах В.В. Крехивского и  М.И. Матийчука~\cite{krix1}, \cite{krix2}, D. Colton~\cite{colton}, O. Arena~\cite{arena}, И.А. Киприянова, В.В. Катрахова, В.М. Ляпина~\cite{kkl},  В.В. Крехивского~\cite{krix3}, И.И. Веренич~\cite{veren}, С.А. Терсенова~\cite{tca}, С.Д.Ивасишена и В.П. Лавренчука~\cite{ivas}, А.Б.Муравник~\cite{mab} и др. Задача Коши для сингулярных параболических уравнений в классах распределений и в классах обобщенных функций типа $S'$  изучались Я.И. Житомирским~\cite{jit}, В.В.Городецким, И.В.Житарюком, В.П.Лавренчуком~\cite{gjl} и др. Однако, начальные и краевые задачи для уравнений с оператором Бесселя высокого порядка, к настоящему времени остаются малоизученными.

Пусть $x = (x_1 ,x_2 ,...,x_n )$  точка $n$ - мерного евклидова  пространства $R^n$ и $R_ + ^n  = \left\{ {x \in R^n: \,\, x_k  > 0, \,\, k = \overline {1,n} } \right\}.$  В данной работе в области $\Omega  = \{ (x,t):x \in R_ + ^n ,\,\,t \in R_ + ^1 \},$ рассмотрим  задачу нахождения классического  решения $u(x,t)$  уравнения
\begin{equation} \label{eq1}
L_\gamma ^m (u) \equiv \left( {\frac{\partial }{{\partial t}} - \Delta _B } \right)^m u(x,t) = f(x,t),\,\,\,(x,t) \in \Omega,
\end{equation}
удовлетворяющее начальным
\begin{equation} \label{ic2}
\left. {\frac{{\partial ^k u}}{{\partial t^k }}} \right|_{t = 0}  = \varphi _k (x), \,\,
x \in R_ + ^n, \,\, k = \overline {0,m - 1}
\end{equation}
и однородным граничным условиям
\begin{equation} \label{bc3}
\left. {\frac{{\partial ^{2k + 1} u}}{{\partial x_j^{2k + 1} }}} \right|_{x_j  = 0}  = 0, \,\,\, t > 0,\,\,\,\,j = \overline {1,n} ,\,\,\,\,\,\,k = \overline {0,m - 1},
\end{equation}
где $\Delta _B  = \sum\limits_{k = 1}^n {B_{\gamma _k }^{x_k } }, $  $B_{\gamma _k }^{x_k }  ={\partial ^2 }/{\partial x_k^2 }+[(2\gamma _k  + 1)/{x_k }](\partial/{\partial x_k })$ - оператор Бесселя, действующий по переменной $x_k; $ $\gamma  = (\gamma _1 ,\gamma _2 ,...,\gamma _n ) \in R^n, $ $\gamma _k  \in R,\,\,\gamma _k  >  - 1/2,\,\,k = \overline {1,n},$ $m$ - натуральное число; $f(x,t),$ $\varphi _k (x),\,\,\,k = \overline {0,m - 1} $ - заданные дифференцируемые функции.

Отметим, что в задачах общей теории дифференциальных уравнений с частными производными, содержащих по одной или нескольким переменным оператор Бесселя, основным аппаратом исследования является соответствующее интегральное преобразование Фурье - Бесселя. 	В отличие от традиционных методов, для решения поставленной задачи применим многомерный  оператор Эрдейи - Кобера дробного порядка. Поэтому сначала рассмотрим некоторые свойства этого оператора.

\section{Многомерный оператор Эрдейи-Кобера дробного порядка}

В теории и приложениях широко используются различные модификации и обобщения классических операторов интегрирования и дифференцирования дробного порядка. К таким модификациям относятся, в частности, операторы Эрдейи - Кобера \cite{skm}.

В работе \cite{ksht1} был введен многомерный обобщенный оператор Эрдейи - Кобера в виде
\[
J_\lambda  \left( \begin{gathered}
  \alpha  \hfill \\
  \eta  \hfill \\
\end{gathered}  \right)f(x) = J_{\lambda _1 ,\lambda _2 , \ldots ,\lambda _n } \left( \begin{gathered}
  \alpha _{1} ,\,\alpha _2 ,\,...\,,\,\alpha _n  \hfill \\
  \eta _1 ,\,\eta _2 ,\,...\,\,,\,\,\eta _n  \hfill \\
\end{gathered}  \right)f(x) =
\]
\[
 = J_{\lambda _1 }^{x_1 } (\eta _1 ,\alpha _1 )J_{\lambda _2 }^{x_2 } (\eta _2 ,\alpha _2 ) \cdots J_{\lambda _n }^{x_n } (\eta _n ,\alpha _n )f(x) =
\]
\[
 = \left[ {\prod\limits_{k = 1}^n {\frac{{2x_k^{ - 2(\alpha _k  + \eta _k )} }}
{{\Gamma (\alpha _k )}}} } \right]\int\limits_0^{x_1 } {\int\limits_0^{x_2 } {...} \int\limits_0^{x_n } {\prod\limits_{k = 1}^n {\left[ {t_k^{2\eta _k  + 1} \left( {x_k^2  - t_k^2 } \right)^{\alpha _k  - 1} } \right.} } }  \times
\]
\begin{equation} \label{moek4}
\times \left. {\bar J_{\alpha _k  - 1} \left( {\lambda \sqrt {x_k^2  - t_k^2 } } \right)} \right]f(t_1 ,\,t_2 ,\,...,\,t_n )dt_1 dt_2 \,...\,dt_n,
\end{equation}
где $\lambda ,\alpha ,\eta  \in R^n, $ $\alpha _k  > 0,\,\,\eta _k  \geqslant  - 1/2,$ $k = \overline {1,\,n};$ $\Gamma (\alpha )$  - гамма - функция  Эйлера, $\bar J_\nu  (z) - $ функция Бесселя -Клиффорда, которая выражается через функции Бесселя $J_\nu  (z),$  по формуле  $\bar J_\nu  (z) = \Gamma (\nu  + 1)(z/2)^{ - \nu } J_\nu  (z),$ $J_{\lambda _k }^{x_k } (\eta _k ,\alpha _k )$ -частный интеграл Эрдейи - Кобера порядка $\alpha _k$  по $k$ -ой переменной:
\[
J_{\lambda _k }^{x_k } (\eta _k ,\alpha _k )f(x) = \frac{{2x_k ^{ - 2(\alpha _k  + \eta _k )} }}
{{\Gamma (\alpha _k )}}\int\limits_0^{x_k } {(x_k^2  - t^2 )^{\alpha _k  - 1} \bar J_{\alpha _k  - 1} \left( {\lambda _k \sqrt {x_k^2  - t^2 } } \right) \times }
\]
\[
 \times t^{2\eta _k  + 1} f(x_1 ,x_2 , \ldots ,x_{k - 1} ,t,x_{k + 1} , \ldots ,x_n )dt.
\]

В указанной работе также исследованы основные свойства оператора  (4) и показано, что обратный оператор имеет вид
\[
J_\lambda ^{ - 1} \left( \begin{gathered}
  \alpha  \hfill \\
  \eta  \hfill \\
\end{gathered}  \right)f(x) = J_{i\lambda } \left( \begin{gathered}
  \,\,\, - \alpha  \hfill \\
  \eta  + \alpha  \hfill \\
\end{gathered}  \right)f(x) =
\]
\[
 = 2^{n - \left| m \right|} \left[ {\prod\limits_{k = 1}^n {\frac{{x_k^{ - 2\eta _k } }}
{{\Gamma (m_k  - \alpha _k )}}} \left( {\frac{1}
{{x_k }}\frac{\partial }
{{\partial x_k }}} \right)^{m_k } } \right]\int\limits_0^{x_1 } {\int\limits_0^{x_2 } {...} \int\limits_0^{x_n } {\prod\limits_{k = 1}^n {\left[ {t_k^{2(\eta _k  + \alpha _k ) + 1} } \right.} } }  \times
\]
\begin{equation} \label{oek5}
 \times \left. {\left( {x_k^2  - t_k^2 } \right)^{m_k  - 1 - \alpha _k } \bar I_{m_k  - 1 - \alpha _k } \left( {\lambda \sqrt {x_k^2  - t_k^2 } } \right)} \right]f(t_1 ,\,t_2 ,\,...,\,t_n )dt_1 dt_2 \,...\,dt_n,
\end{equation}
где $\alpha _k  > 0,\,\,m_k  = [\alpha _k ] + 1,$ $\eta _k  \geqslant  - 1/2,$ $k = \overline {1,\,n}, $ $\bar I_\nu  (z) = \Gamma (\nu  + 1)(z/2)^{ - \nu } I_\nu  (z),$ $I_\nu  (z)$ - функция Бесселя мнимого аргумента;  $m = (m_1 ,\,m_2 ,\,...\,,\,m_n )$  - мультииндекс, а $\left| m \right| = m_1  + m_2  + \,...\, + m_n $  его длина.

Учитывая $\bar J_\nu  (0) = 1$  в пределе при $\lambda _k  \to 0,\,\,\,k = \overline {1,n}, $ получим
\[
J_0 \left( \begin{gathered}
  \alpha  \hfill \\
  \eta  \hfill \\
\end{gathered}  \right)f(x) = J_{0,0, \ldots ,0} \left( \begin{gathered}
  \alpha _{1} ,\,\alpha _2 ,\,...\,,\,\alpha _n  \hfill \\
  \eta _1 ,\,\eta _2 ,\,...\,\,,\,\,\eta _n  \hfill \\
\end{gathered}  \right)f(x) =
\]
\begin{equation} \label{oek6}
 = \prod\limits_{k = 1}^n {\left[ {\frac{{2x_k^{ - 2(\alpha _k  + \eta _k )} }}
{{\Gamma (\alpha _k )}}} \right]} \int\limits_0^{x_1 } {\int\limits_0^{x_2 } {...} \int\limits_0^{x_n } {\prod\limits_{k = 1}^n {\left[ {t_k^{2\eta _k  + 1} \left( {x_k^2  - t_k^2 } \right)^{\alpha _k  - 1} } \right]f(t)dt_1 dt_2 \,...\,dt_n } } },
\end{equation}

Этот оператор является многомерным аналогом обычного (не обобщенного) оператора Эрдейи-Кобера. В этом случае обратный оператор (5) примет следющий вид:
\[
J_0^{ - 1} \left( \begin{gathered}
  \alpha  \hfill \\
  \eta  \hfill \\
\end{gathered}  \right)f(x) = 2^{n - \left| m \right|} \left[ {\prod\limits_{k = 1}^n {\frac{{x_k^{ - 2\eta _k } }}
{{\Gamma (m_k  - \alpha _k )}}} \left( {\frac{1}
{{x_k }}\frac{\partial }
{{\partial x_k }}} \right)^{m_k } } \right]\int\limits_0^{x_1 } {\int\limits_0^{x_2 } {...} \int\limits_0^{x_n } {\prod\limits_{k = 1}^n {\left[ {t_k^{2(\eta _k  + \alpha _k ) + 1} } \right.} } }  \times
\]
\begin{equation} \label{ooek7}
 \times \left. {\left( {x_k^2  - t_k^2 } \right)^{m_k  - 1 - \alpha _k } } \right]f(t_1 ,\,t_2 ,\,...,\,t_n )dt_1 dt_2 \,...\,dt_n.
\end{equation}

Кроме того  в работе~\cite{ksht1} доказана следующая теорема.

\begin{theorem} \label{th1}
Пусть $\alpha _k  > 0,\,\,\,\eta _k  \geqslant  - 1/2,\,\,k = \overline {1,n};$ $f(x) \in C^2 (\Omega ^n );$
функции $x_k^{2\eta _k  + 1} B_{\eta _k }^{x_k } f(x)$  интегрируемы при $x_k  \to 0$ и $\mathop {\lim }\limits_{x_k  \to 0} x_k^{2\eta _k  + 1} f_{x_k } (x) = 0,$  $k = \overline {1,\,n}.$

Тогда
\[
(B_{\eta _k  + \alpha _k }^{x_k }  + \lambda _k^2 )J_\lambda  \left( \begin{gathered}
  \alpha  \hfill \\
  \eta  \hfill \\
\end{gathered}  \right)f(x) = J_\lambda  \left( \begin{gathered}
  \alpha  \hfill \\
  \eta  \hfill \\
\end{gathered}  \right)B_{\eta _k }^{x_k } f(x), \,\, k = \overline {1,n},
\]
в частности, если $\lambda _k  = 0,\,\,k = \overline {1,n}, $ тогда
\[
B_{\eta _k  + \alpha _k }^{x_k } J_0 \left( \begin{gathered}
  \alpha  \hfill \\
  \eta  \hfill \\
\end{gathered}  \right)f(x) = J_0 \left( \begin{gathered}
  \alpha  \hfill \\
  \eta  \hfill \\
\end{gathered}  \right)B_{\eta _k }^{x_k } f(x), \,\, k = \overline {1,n},
\]
где  $\Omega ^n  = \prod\limits_{k = 1}^n {(0,\,b_k )}  = (0,\,b_1 ) \times (0,\,b_2 ) \times \,...\, \times (0,\,b_n )$ - декартово произведение, $b_k  > 0,\,\,k = \overline {1,n}.$
\end{theorem}

Из этой теоремы вытекает:
\begin{corollary} \label{cor1}
Пусть выполнены условия теоремы 1.  Тогда
\[
\sum\limits_{k = 1}^n {\left[ {B_{\eta _k  + \alpha _k }^{x_k }  + \lambda _k^2 } \right]} J_\lambda  \left( \begin{gathered}
  \alpha  \hfill \\
  \eta  \hfill \\
\end{gathered}  \right)f(x) = J_\lambda  \left( \begin{gathered}
  \alpha  \hfill \\
  \eta  \hfill \\
\end{gathered}  \right)\sum\limits_{k = 1}^n {\left[ {B_{\eta _k }^{x_k } } \right]f(x)},
\]
в частности, если $\eta _k  =  - 1/2,$ $k = \overline {1,n},$ тогда
\[
\sum\limits_{k = 1}^n {\left[ {B_{\alpha _k  - 1/2}^{x_k }  + \lambda _k^2 } \right]} J_\lambda  \left( \begin{gathered}
  \,\,\,\,\,\,\alpha  \hfill \\
   - 1/2 \hfill \\
\end{gathered}  \right)f(x) = J_\lambda  \left( \begin{gathered}
  \,\,\,\,\,\,\alpha  \hfill \\
   - 1/2 \hfill \\
\end{gathered}  \right)\Delta f(x),
\]
где  $\Delta f(x) \equiv \sum\limits_{k = 1}^n {\left[ {\partial ^2 f(x)/\partial x_k^2 } \right]}$
- многомерный оператор Лапласа.
\end{corollary}

Далее, докажем некоторые свойства многомерного обобщенного оператора Эрдейи-Кобера, которые применяются при решении поставленной задачи.

\begin{theorem} \label{th2}
Пусть $\alpha _k  > 0,\,\,\,\eta _k  \geqslant  - 1/2,\,\,\,k = \overline {1,n},$ $f(x) \in C^{2n} (\Omega ^n ),$ $x_k^{2\eta _k  + 1} B_{\eta _k }^{x_k } f(x)$  - интегрируемы в окрестности $x_k  = 0$
 и $\mathop {\lim }\limits_{x_k  \to 0} x_k^{2\eta _k  + 1} f_{x_k } (x) = 0,$   $k = \overline {1,\,n}. $

Тогда  имеет место равенство
\[
\prod\limits_{k = 1}^n {(B_{\eta _k  + \alpha _k }^{x_k }  + \lambda _k^2 )} J_\lambda  \left( \begin{gathered}
  \alpha  \hfill \\
  \eta  \hfill \\
\end{gathered}  \right)f(x) = J_\lambda  \left( \begin{gathered}
  \alpha  \hfill \\
  \eta  \hfill \\
\end{gathered}  \right)\prod\limits_{k = 1}^n {B_{\eta _k }^{x_k } f(x)},
\]
в частности, если $\eta _k  =  - 1/2,$ $k = \overline {1,n},$ тогда
\[
\prod\limits_{k = 1}^n {\left[ {B_{\alpha _k  - (1/2)}^{x_k }  + \lambda _k^2 } \right]} J_\lambda  \left( \begin{gathered}
 \,\,\,\,\,\, \alpha  \hfill \\
  - 1/2  \hfill \\
\end{gathered}  \right)f(x) = J_\lambda  \left( \begin{gathered}
 \,\,\,\,\,\, \alpha  \hfill \\
  - 1/2 \hfill \\
\end{gathered}  \right)\frac{{\partial ^{2n} f(x)}}
{{\partial x_1^2 \partial x^2  \ldots \partial x_n^2 }}.
\]
\end{theorem}
Доказательство теоремы 2  аналогично доказательству теоремы 1.

	Пусть $\left[ {B_{\eta _k }^{x_k } } \right]^0  = E,$ где $E$ -единичный оператор,  $\left[ {B_{\eta _k }^{x_k } } \right]^{m_k }  =$ $\left[ {B_{\eta _k }^{x_k } } \right]^{m_k  - 1} \left[ {B_{\eta _k }^{x_k } } \right] =$ $\left[ {B_{\eta _k }^{x_k } } \right]\left[ {B_{\eta _k }^{x_k } } \right] \cdots \left[ {B_{\eta _k }^{x_k } } \right]$  - $m_k $  -ая степень оператора $B_{\eta _k }^{x_k } ,\,\,k = \overline {0,n}.$

\begin{theorem} \label{th3}
Пусть $\alpha _k  > 0,\,\,\,\eta _k  \geqslant  - 1/2;$ $f(x) \in C^{2m_0 } (\Omega ^n );$ $x_k^{2\eta _k  + 1} \left[ {B_{\eta _k }^{x_k } } \right]^{p_k  + 1} f(x)$  интегрируемы при $x_k  \to 0$  и $\mathop {\lim }\limits_{x_k  \to 0} x_k^{2\eta _k  + 1} (\partial/{\partial x_k }) \left[ {B_{\eta _k }^{x_k } } \right]^{p_k } f(x) = 0,$ $p_k  = \overline {0,m_k  - 1},$ $k = \overline {1,\,n}.$
Тогда
\begin{equation} \label{koek8}
\left[ {B_{\eta _k  + \alpha _k }^{x_k }  + \lambda _k^2 } \right]^{m_k } J_\lambda  \left( \begin{gathered}
  \alpha  \hfill \\
  \eta  \hfill \\
\end{gathered}  \right)f(x) = J_\lambda  \left( \begin{gathered}
  \alpha  \hfill \\
  \eta  \hfill \\
\end{gathered}  \right)\left[ {B_{\eta _k }^{x_k } } \right]^{m_k } f(x),\,\,k = \overline {1,n},
\end{equation}
где $m_0  = \max \{ m_1 ,m_2 , \ldots ,m_n \}. $
\end{theorem}

Заметим, что теорема 3 верна и в том случае, когда некоторые или все $\lambda _k  = 0,\,\,k = \overline {1,n}.$

\begin{proof} Теорему 3 можно доказать методом математической индукции по $m_k ,\,\,k = \overline {1,n}. $ Произвольно фиксируем $k \in N,$ где $N$-множество натуральных чисел. Доказательство формулы (8) при фиксированном $k$  и $m_k  = 1$ приведено в теореме 1. Предположим, что равенство (8) имеет место при $m_k  = l_k $  и докажем, что оно справедливо при $m_k  = l_k  + 1.$
Из равенства
\[
\left[ {B_{\eta _k  + \alpha _k }^{x_k }  + \lambda _k^2 } \right]^{l_k  + 1} J_\lambda  \left( \begin{gathered}
  \alpha  \hfill \\
  \eta  \hfill \\
\end{gathered}  \right)f(x) = \left[ {B_{\eta _k  + \alpha _k }^{x_k }  + \lambda _k^2 } \right]\left[ {B_{\eta _k  + \alpha _k }^{x_k }  + \lambda _k^2 } \right]^{l_k } J_\lambda  \left( \begin{gathered}
  \alpha  \hfill \\
  \eta  \hfill \\
\end{gathered}  \right)f(x)
\]
по предположению индукции при выполнении условий теоремы 3, имеем
\[
\left[ {B_{\eta _k  + \alpha _k }^{x_k }  + \lambda _k^2 } \right]\left[ {B_{\eta _k  + \alpha _k }^{x_k }  + \lambda _k^2 } \right]^{l_k } J_\lambda  \left( \begin{gathered}
  \alpha  \hfill \\
  \eta  \hfill \\
\end{gathered}  \right)f(x) = \left[ {B_{\eta _k  + \alpha _k }^{x_k }  + \lambda _k^2 } \right]J_\lambda  \left( \begin{gathered}
  \alpha  \hfill \\
  \eta  \hfill \\
\end{gathered}  \right)\left[ {B_{\eta _k }^{x_k } } \right]^{l_k } f(x).
\]
При выполнении условий $\mathop {\lim }\limits_{x_k  \to 0} x_k^{2\eta _k  + 1} (\partial/{\partial x_k })\left[ {B_{\eta _k }^{x_k } } \right]^{l_k } f(x) = 0,\,\,\,k = \overline {1,n},$ вытекает условия применимости теоремы 1 к функциям $\left[ {B_{\eta _k }^{x_k } } \right]^{l_k } f(x).$ Поэтому в последнем равенстве применяя теорему 1,  получим справедливость формулы (8).
\end{proof}
Из теоремы 3 вытекает следующее
\begin{corollary} \label{cor2} Пусть выполнены условия теоремы 3. Тогда
\[
\sum\limits_{k = 1}^n {\left[ {B_{\eta _k  + \alpha _k }^{x_k }  + \lambda _k^2 } \right]^{m_k } } J_\lambda  \left( \begin{gathered}
  \alpha  \hfill \\
  \eta  \hfill \\
\end{gathered}  \right)f(x) = J_\lambda  \left( \begin{gathered}
  \alpha  \hfill \\
  \eta  \hfill \\
\end{gathered}  \right)\sum\limits_{k = 1}^n {\left[ {B_{\eta _k }^{x_k } } \right]^{m_k } f(x)},
\]
кроме того, если $f(x) \in C^{2\left| m \right|} (\Omega ^n ),$ тогда
\begin{equation} \label{koek9}
\prod\limits_{k = 1}^n {\left[ {B_{\eta _k  + \alpha _k }^{x_k }  + \lambda _k^2 } \right]^{m_k } } J_\lambda  \left( \begin{gathered}
  \alpha  \hfill \\
  \eta  \hfill \\
\end{gathered}  \right)f(x) = J_\lambda  \left( \begin{gathered}
  \alpha  \hfill \\
  \eta  \hfill \\
\end{gathered}  \right)\prod\limits_{k = 1}^n {\left[ {B_{\eta _k }^{x_k } } \right]^{m_k } f(x)}.
\end{equation}
\end{corollary}

\begin{theorem} \label{th4} Пусть $\alpha _k  > 0,\,\,\,\eta _k  \geqslant  - 1/2,\,\,k = \overline {1,n},$
$q \in N;$ $f(x) \in C^{2q} (\Omega ^n );$ функции $x_k^{2\eta _k  + 1} \left[ {B_{\eta _k }^{x_k } } \right]^{l + 1} f(x)$ интегрируемы в нуле и $\mathop {\lim }\limits_{x_k  \to 0} x_k^{2\eta _k  + 1} (\partial/{\partial x_k })\left[ {B_{\eta _k }^{x_k } } \right]^l f(x) = 0,$ $l = \overline {0,q - 1},$ $k = \overline {1,\,n}, $ Тогда имеет место равенство
\[
\left[ {\sum\limits_{k = 1}^n {\left( {B_{\eta _k  + \alpha _k }^{x_k }  + \lambda _k^2 } \right)} } \right]^q J_\lambda  \left( \begin{gathered}
  \alpha  \hfill \\
  \eta  \hfill \\
\end{gathered}  \right)f(x) = J_\lambda  \left( \begin{gathered}
  \alpha  \hfill \\
  \eta  \hfill \\
\end{gathered}  \right)\left[ {\sum\limits_{k = 1}^n {B_{\eta _k }^{x_k } } } \right]^q f(x).
\]
\end{theorem}

Данная теорема доказывается с использованием полиномиальной  формулы $\left[ {\sum\limits_{k = 1}^n {\left( {B_{\eta _k  + \alpha _k }^{x_k }  + \lambda _k^2 } \right)} } \right]^q  = \sum\limits_{\left| m \right| = q} {\frac{{q!}}{{m!}}\prod\limits_{k = 1}^n {\left( {B_{\eta _k  + \alpha _k }^{x_k }  + \lambda _k^2 } \right)^{m_k } } } $ и с применением равенства (9), где $m! = m_1 !m_2 ! \ldots m_n !.$

\begin{corollary} \label{cor3}
 Пусть выполнены условия теоремы 4. Тогда при $\eta _k  =  - 1/2,\,\,k = \overline {1,n}, $ справедливо равенство
\[
\left[ {\sum\limits_{k = 1}^n {\left( {B_{\alpha _k  - 1/2}^{x_k }  + \lambda _k^2 } \right)} } \right]^q J_\lambda  \left( \begin{gathered}
  \,\,\,\,\,\alpha  \hfill \\
   - 1/2 \hfill \\
\end{gathered}  \right)f(x) = J_\lambda  \left( \begin{gathered}
  \,\,\,\,\,\alpha  \hfill \\
   - 1/2 \hfill \\
\end{gathered}  \right)\Delta ^q f(x),
\]
в частности, при $\lambda _k  = 0,$ верно равенство
\[
\Delta _B^q J_0 \left( \begin{gathered}
  \,\,\,\,\,\alpha  \hfill \\
   - 1/2 \hfill \\
\end{gathered}  \right)f(x) = J_0 \left( \begin{gathered}
  \,\,\,\,\,\alpha  \hfill \\
   - 1/2 \hfill \\
\end{gathered}  \right)\Delta ^q f(x),
\]
где  $\Delta _B  \equiv \sum\limits_{k = 1}^n {B_{\alpha _k  - 1/2}^{x_k }  = } \sum\limits_{k = 1}^n {\left( {\dfrac{{\partial ^2 }}{{\partial x_k^2 }} + \dfrac{{2\alpha _k }}{{x_k }}\dfrac{\partial }
{{\partial x_k }}} \right)}. $
\end{corollary}

Пусть $L^{(y)} $ - не зависящий от  переменной $x = (x_1 ,\,x_2 ,\, \ldots ,\,x_n )$  линейный дифференциальный оператор порядка $l \in N$  по переменной $y = (y_1 ,y_2 ,\, \ldots ,\,y_s ) \in R^s. $

\begin{theorem} \label{th5}
Пусть $\alpha _k  > 0,\,\, \eta _k  \geqslant  - 1/2, \,\, k = \overline {1,n},$ $q \in N;$ $f(x,y) \in C_{x,y}^{2q,lq} (\Omega ^n  \times \Omega ^s ),$ $x_k^{2\eta _k  + 1} \left[ {B_{\eta _k }^{x_k } } \right]^{j + 1} f(x,y)$  интегрируемы в окрестности начала координат и $\mathop {\lim }\limits_{x_k  \to 0} x_k^{2\eta _k  + 1} (\partial/{\partial x_k })\left[ {B_{\eta _k }^{x_k } } \right]^j f(x,y) = 0,$ $j = \overline {0,q - 1},$ $k = \overline {1,\,n}. $   Тогда
\[
\left[ {L^{(y)}  \pm \sum\limits_{k = 1}^n {\left( {B_{\eta _k  + \alpha _k }^{x_k }  + \lambda _k^2 } \right)} } \right]^q J_\lambda ^{(x)} \left( \begin{gathered}
  \alpha  \hfill \\
  \eta  \hfill \\
\end{gathered}  \right)f(x,y) =
\]
\[
=J_\lambda ^{(x)} \left( \begin{gathered}
  \alpha  \hfill \\
  \eta  \hfill \\
\end{gathered}  \right)\left[ {L^{(y)}  \pm \sum\limits_{k = 1}^n {B_{\eta _k }^{x_k } } } \right]^q f(x,y),
\]
где верхние индексы в операторах означают переменные, по которым действуют эти операторы.
\end{theorem}

\begin{proof} Используя биномиальную формулу, получим
\[
\left[ {L^{(y)}  \pm \sum\limits_{k = 1}^n {\left( {B_{\eta _k  + \alpha _k }^{x_k }  + \lambda _k^2 } \right)} } \right]^q J_\lambda ^{(x)} \left( \begin{gathered}
  \alpha  \hfill \\
  \eta  \hfill \\
\end{gathered}  \right)f(x,y) =
\]
\[
 = \sum\limits_{j = 0}^q {C_q^j ( \pm 1)^j \left( {L^{(y)} } \right)^{q - j} \left[ {\sum\limits_{k = 1}^n {\left( {B_{\eta _k  + \alpha _k }^{x_k }  + \lambda _k^2 } \right)} } \right]^j } J_\lambda ^{(x)} \left( \begin{gathered}
  \alpha  \hfill \\
  \eta  \hfill \\
\end{gathered}  \right)f(x,y).
\]

	Далее применяя теорему 4,  имеем
\[
J_\lambda ^{(x)} \left( \begin{gathered}
  \alpha  \hfill \\
  \eta  \hfill \\
\end{gathered}  \right)\sum\limits_{j = 0}^q {\left( {\begin{array}{*{20}c}
   q  \\
   j  \\

 \end{array} } \right)( \pm 1)^j \left( {L^{(y)} } \right)^{q - j} \left[ {\sum\limits_{k = 1}^n {B_{\eta _k }^{x_k } } } \right]^j } f(x,y) =
\]
\[
= J_\lambda ^{(x)} \left( \begin{gathered}
  \alpha  \hfill \\
  \eta  \hfill \\
\end{gathered}  \right)\left[ {L^{(y)}  \pm \sum\limits_{k = 1}^n {B_{\eta _k }^{x_k } } } \right]^q f(x,y).
\]
Теорема 5 доказана.
\end{proof}

\begin{corollary} \label{cor4}
Пусть выполнены условия теоремы 5.  Если $x = (x_1 ,x_2 , \ldots ,x_\omega  ),$ $y = (x_{\omega  + 1} ,x_{\omega  + 2} , \ldots ,x_{\omega  + \sigma } ),\,\,\omega  + \sigma  = n,$ $L^{(y)}  =  - \sum\limits_{k = \omega  + 1}^{\omega  + \sigma } {\left( {B_{\eta _k  + \alpha _k }^{x_k }  + \lambda _k^2 } \right)}, $ тогда
\[
\left[ {\sum\limits_{k = 1}^\omega  {\left( {B_{\eta _k  + \alpha _k }^{x_k }  + \lambda _k^2 } \right)}  - \sum\limits_{k = \omega  + 1}^{\omega  + \sigma } {\left( {B_{\eta _k  + \alpha _k }^{x_k }  + \lambda _k^2 } \right)} } \right]^q J_\lambda  \left( \begin{gathered}
  \alpha  \hfill \\
  \eta  \hfill \\
\end{gathered}  \right)f(x) =
\]
\[
 = J_\lambda  \left( \begin{gathered}
  \alpha  \hfill \\
  \eta  \hfill \\
\end{gathered}  \right)\left[ {\sum\limits_{k = 1}^\omega  {B_{\eta _k }^{x_k } }  - \sum\limits_{k = \omega  + 1}^{\omega  + \sigma } {B_{\eta _k }^{x_k } } } \right]^q f(x),\,\,\,\,\,\,\,\omega  + \sigma  = n,
\]
в частности, если $L^{(y)}  \equiv 0,$ тогда
\[
\left[ {\sum\limits_{k = 1}^n {\left( {B_{\eta _k  + \alpha _k }^{x_k }  + \lambda _k^2 } \right)} } \right]^q J_\lambda  \left( \begin{gathered}
  \alpha  \hfill \\
  \eta  \hfill \\
\end{gathered}  \right)f(x) = J_\lambda  \left( \begin{gathered}
  \alpha  \hfill \\
  \eta  \hfill \\
\end{gathered}  \right)\left[ {\sum\limits_{k = 1}^n {B_{\eta _k }^{x_k } } } \right]^q f(x).
\]
\end{corollary}

Пусть $[D_{\eta _k }^{x_k } ]^0  = E,$ $D_{\eta _k }^{x_k }  \equiv x_k^{ - 2\eta _k } \left( {\dfrac{1}
{{x_k }}\dfrac{\partial }{{\partial x_k }}} \right)x_k^{2\eta _k }, $  $[D_{\eta _k }^{x_k } ]^{m_k }  = [D_{\eta _k }^{x_k } ]^{m_k  - 1} D_{\eta _k }^{x_k }  = $ $ = D_{\eta _k }^{x_k } D_{\eta _k }^{x_k } ...D_{\eta _k }^{x_k } $  - $m_k $ - я степень оператора $D_{\eta _k }^{x_k }, $ которая представима в виде $[D_{\eta _k }^{x_k } ]^{m_k }  = x_k^{ - 2\eta _k } \left( {\dfrac{1}{{x_k }}\dfrac{\partial }{{\partial x_k }}} \right)^{m_k } x_k^{2\eta _k }, $ $m_k $ -неотрицательные целые числа, $k = \overline {1,n}. $

\begin{theorem} \label{th6}	
Если $\alpha _k  > 0,\,\,\eta _k  \geqslant  - (1/2),\,\,k = \overline {1,n}, $ $f(x) \in C^{m_0 } (\Omega ^n ),$ функции $x_{x_k }^{2\eta _k  + 1} [D_{\eta _k }^{x_k } ]^{\,l_k  + 1} f(x)$  - интегрируемы при $x_k  \to 0$ и  $\mathop {\lim }\limits_{x_k  \to 0} x_k^{2\eta _k } [D_{\eta _k }^{x_k } ]^{\,l_k } f(x) = 0,\,\,\,l_k  = \overline {0,m_k  - 1} ,\,\,\,k = \overline {1,n}, $  то
\begin{equation} \label{koek10}
[D_{\eta _k  + \alpha _k }^{x_k } ]^{m_k } J_\lambda  \left( \begin{gathered}
  \alpha  \hfill \\
  \eta  \hfill \\
\end{gathered}  \right)f(x) = J_\lambda  \left( \begin{gathered}
  \alpha  \hfill \\
  \eta  \hfill \\
\end{gathered}  \right)[D_{\eta _k }^{x_k } ]^{m_k } f(x),\,\,\,k = \overline {1,n},
\end{equation}
где  $m_0  = \max \{ m_1 ,m_2 , \ldots ,m_n \}$
\end{theorem}

\begin{proof} Эта теорема также доказывается с применением метода математической индукции по $m_k ,\,\,k = \overline {1,n}. $ Произвольно фиксируем $k \in N.$ Доказательство формулы (10) при  $m_k  = 1,$ $k = \overline {1,n} $  приведены в работах~\cite{Kar1}, \cite{Kar2}, согласно которой имеем
\begin{equation} \label{koek11}
[D_{\eta _k  + \alpha _k }^{x_k } ]J_\lambda  \left( \begin{gathered}
  \alpha  \hfill \\
  \eta  \hfill \\
\end{gathered}  \right)f(x) = J_\lambda  \left( \begin{gathered}
  \alpha  \hfill \\
  \eta  \hfill \\
\end{gathered}  \right)[D_{\eta _k }^{x_k } ]f(x),\,\,\,k = \overline {1,n}.
\end{equation}

Предположим, что равенство (10) имеет место при $m_k  = l_k $ и докажем, что оно справедливо при $m_k  = l_k  + 1.$
\begin{equation} \label{koek12}
[D_{\eta _k  + \alpha _k }^{x_k } ]^{l_k  + 1} J_\lambda  \left( \begin{gathered}
  \alpha  \hfill \\
  \eta  \hfill \\
\end{gathered}  \right)f(x) = [D_{\eta _k  + \alpha _k }^{x_k } ][D_{\eta _k  + \alpha _k }^{x_k } ]^{l_k } J_\lambda  \left( \begin{gathered}
  \alpha  \hfill \\
  \eta  \hfill \\
\end{gathered}  \right)f(x).
\end{equation}

По предположению индукции при выполнении условий теоремы 6, имеем
\[
[D_{\eta _k  + \alpha _k }^{x_k } ]^{l_k } J_\lambda  \left( \begin{gathered}
  \alpha  \hfill \\
  \eta  \hfill \\
\end{gathered}  \right)f(x) = J_\lambda  \left( \begin{gathered}
  \alpha  \hfill \\
  \eta  \hfill \\
\end{gathered}  \right)[D_{\eta _k }^{x_k } ]^{l_k } f(x).
\]

Тогда равенство (12) примет вид
\[
[D_{\eta _k  + \alpha _k }^{x_k } ]^{l_k  + 1} J_\lambda  \left( \begin{gathered}
  \alpha  \hfill \\
  \eta  \hfill \\
\end{gathered}  \right)f(x) = [D_{\eta _k  + \alpha _k }^{x_k } ]J_\lambda  \left( \begin{gathered}
  \alpha  \hfill \\
  \eta  \hfill \\
\end{gathered}  \right)[D_{\eta _k }^{x_k } ]^{l_k } f(x).
\]

Далее, при выполнении условий $\mathop {\lim }\limits_{x_k  \to 0} x_k^{2\eta _k } [D_{\eta _k }^{x_k } ]^{\,l_k } f(x) = 0,$ применяя формулу (11) к функциям $[D_{\eta _k }^{x_k } ]^{l_k } f(x),$  получим справедливость формулы (10).
\end{proof}

\begin{corollary} \label{cor5}  Пусть выполнены условия теоремы 6, тогда
\[
\prod\limits_{k = 1}^n {[D_{\eta _k  + \alpha _k }^{x_k } ]^{m_k } J_\lambda  \left( \begin{gathered}
  \alpha  \hfill \\
  \eta  \hfill \\
\end{gathered}  \right)f(x)}  = J_\lambda  \left( \begin{gathered}
  \alpha  \hfill \\
  \eta  \hfill \\
\end{gathered}  \right)\prod\limits_{k = 1}^n {[D_{\eta _k }^{x_k } ]^{m_k } f(x)}.
\]
\end{corollary}

\begin{theorem} \label{th7}	
Пусть $0 < \alpha _k  < 1,\,\,\,\eta _k  \geqslant  - 1/2,\,\,k = \overline {1,n},$ $p \in N;$ $g(x) \in C^{2p} (\Omega ^n );$ $\dfrac{\partial }{{\partial x_k }}\left[ {B_{\eta _k + \alpha_k }^{x_k } } \right]^l g(x)$   интегрируемы в нуле и $\mathop {\lim }\limits_{x_k  \to 0} x_k^{2(\eta _k  + \alpha _k ) + 1} \dfrac{\partial }
{{\partial x_k }}\left[ {B_{\eta _k  + \alpha _k }^{x_k } } \right]^l g(x) = 0,$ $l = \overline {0,p - 1},$ $k = \overline {1,\,n}. $
	Тогда имеет место равенство
\[
\left[ {B_{\eta _k }^{x_k }  - \lambda _k^2 } \right]^p J_\lambda ^{ - 1} \left( \begin{gathered}
  \alpha  \hfill \\
  \eta  \hfill \\
\end{gathered}  \right)g(x) = J_\lambda ^{ - 1} \left( \begin{gathered}
  \alpha  \hfill \\
  \eta  \hfill \\
\end{gathered}  \right)\left[ {B_{\eta _k  + \alpha _k }^{x_k } } \right]^p g(x),\,\,k = \overline {1,n},
\]
или
\[
\left[ {B_{\eta _k }^{x_k } } \right]^p J_\lambda ^{ - 1} \left( \begin{gathered}
  \alpha  \hfill \\
  \eta  \hfill \\
\end{gathered}  \right)g(x) = J_\lambda ^{ - 1} \left( \begin{gathered}
  \alpha  \hfill \\
  \eta  \hfill \\
\end{gathered}  \right)\left[ {B_{\eta _k  + \alpha _k }^{x_k }  + \lambda _k^2 } \right]^p g(x),\,\,\,k = \overline {1,n},
\]
в частности, если $\lambda _k  = 0,$ тогда
\[
\left[ {B_{\eta _k }^{x_k } } \right]^p J_0^{ - 1} \left( \begin{gathered}
  \alpha  \hfill \\
  \eta  \hfill \\
\end{gathered}  \right)g(x) = J_0^{ - 1} \left( \begin{gathered}
  \alpha  \hfill \\
  \eta  \hfill \\
\end{gathered}  \right)\left[ {B_{\eta _k  + \alpha _k }^{x_k } } \right]^p g(x),\,\,\,k = \overline {1,n}.
\]
\end{theorem}

Доказательство теоремы аналогично доказательству теоремы 3.

\begin{corollary} \label{cor6}
 Пусть выполнены условия теоремы 7, тогда
\[
\left[ {\sum\limits_{k = 1}^n {\left( {B_{\eta _k }^{x_k }  - \lambda _k^2 } \right)} } \right]^p J_\lambda ^{ - 1} \left( \begin{gathered}
  \alpha  \hfill \\
  \eta  \hfill \\
\end{gathered}  \right)g(x) = J_\lambda ^{ - 1} \left( \begin{gathered}
  \alpha  \hfill \\
  \eta  \hfill \\
\end{gathered}  \right)\left[ {\sum\limits_{k = 1}^n {B_{\eta _k  + \alpha _k }^{x_k } } } \right]^p g(x),
\]
или
\[
\left[ {\sum\limits_{k = 1}^n {B_{\eta _k }^{x_k } } } \right]^p J_\lambda ^{ - 1} \left( \begin{gathered}
  \alpha  \hfill \\
  \eta  \hfill \\
\end{gathered}  \right)g(x) = J_\lambda ^{ - 1} \left( \begin{gathered}
  \alpha  \hfill \\
  \eta  \hfill \\
\end{gathered}  \right)\left[ {\sum\limits_{k = 1}^n {\left( {B_{\eta _k  + \alpha _k }^{x_k }  + \lambda _k^2 } \right)} } \right]^p g(x).
\]
\end{corollary}

	Если выполнены условия $\mathop {\lim }\limits_{x_k  \to 0} x_k^{2\alpha _k } \dfrac{\partial }
{{\partial x_k }}\left[ {B_{\alpha _k  - (1/2)}^{x_k } } \right]^l g(x) = 0,$ $l = \overline {0,p - 1},$ $k = \overline {1,\,n}, $ то из последнего равенства при $\lambda _k  = 0,\,\,\,\eta _k  =  - (1/2),$  $k = \overline {0,m - 1}, $ следует справедливость равенства
\begin{equation} \label{koek13}
\Delta ^p J_0^{ - 1} \left( \begin{gathered}
  \,\,\,\,\,\alpha  \hfill \\
   - 1/2 \hfill \\
\end{gathered}  \right)g(x) = J_0^{ - 1} \left( \begin{gathered}
  \,\,\,\,\,\alpha  \hfill \\
   - 1/2 \hfill \\
\end{gathered}  \right)\Delta _B^p g(x)
\end{equation}
где  $\Delta ^p  = \left[ {\sum\limits_{k = 1}^n {\dfrac{{\partial ^2 }}{{\partial x_k^2 }}} } \right]^p$
- $p$ -ая степень многомерного оператора Лапласа, а
\[
\Delta _B^p  = \left[ {\sum\limits_{k = 1}^n {\left( {B_{\alpha _k  - (1/2)}^{x_k } } \right)} } \right]^p  = \left[ {\sum\limits_{k = 1}^n {\left( {\frac{{\partial ^2 }}
{{\partial x_k^2 }} + \frac{{2\alpha _k }}
{{x_k }}\frac{\partial }
{{\partial x_k }}} \right)} } \right]^p.
\]

Доказанные теоремы позволяют сводить многомерные уравнения высокого порядка с сингулярными коэффициентами к полигармоническим, поликалорическим и поливолновым уравнениям и тем самым поставить и исследовать корректные  начальные и граничные задачи для таких уравнений.

\section{Решение поставленной задачи в случае однородного уравнения}

Пусть  существует решение однородного уравнения $L_\gamma ^m (u) = 0,$ удовлетворяющее условиям (2) и (3). Это решение ищем в виде
\begin{equation} \label{sp14}
u(x,t) = J_0^{(x)} \left( \begin{gathered}
  \alpha  \hfill \\
  \eta  \hfill \\
\end{gathered}  \right)U(x,t),
\end{equation}
где  $\alpha ,\eta  \in R^n, $ причем $\alpha _k  = \gamma _k  + (1/2) > 0,$ $\eta _k  =  - (1/2),\,\,k = \overline {1,n}, $ а $U(x,t)$  -неизвестная и достаточное число раз дифференцируемая функция, а $J_0^{(x)} \left( \begin{gathered}
  \alpha  \hfill \\
  \eta  \hfill \\
\end{gathered}  \right)
$  - многомерный оператор Эрдейи-Кобера дробного порядка (6), действующий по переменной $x \in R^n. $

Подставляя (14) в граничные условия (3), а затем в уравнение (1) и начальные условия (2), и используя теорему 5 при $L^{(t)} \equiv \partial /\partial t,$  получим следующую задачу  нахождения решения  $U(x,t)$   уравнения
\begin{equation} \label{eq15}
\left( {\frac{\partial }{{\partial t}} - \Delta } \right)^m U(x,t) = 0,\,\,(x,t) \in \Omega,
\end{equation}
удовлетворяющим  начальным
\begin{equation} \label{ic16}
\left. {\frac{{\partial ^k U}}{{\partial t^k }}} \right|_{t = 0}  = \Phi _k (x), \,\,
x \in R^n ,\,\,\,k = \overline {0,m - 1},
\end{equation}
и однородным граничным условиям
\begin{equation} \label{bc17}
\left. {\frac{{\partial ^{2k + 1} U}}{{\partial x_j^{2k + 1} }}} \right|_{x_j  = 0}  = 0,
t > 0,\,\,\,\,j = \overline {1,n} ,\,\,\,\,\,\,k = \overline {0,m - 1},
\end{equation}
где  $ \Phi _k (x) = J_0^{ - 1} \left( \begin{gathered} \alpha  \hfill \\  \eta  \hfill \\ \end{gathered}  \right)\varphi _k (x), $ $\eta _k  =  - (1/2),$ $(k = \overline {0,m - 1} ),$ $J_0^{ - 1} \left( \begin{gathered}  \alpha  \hfill \\   \eta  \hfill \\ \end{gathered}  \right)$ - обратный оператор (7).

	Учитывая граничные условия (17), продолжим функции $\Phi _k (x)$  четным образом на $x_k  < 0,$ $(k = \overline {0,m - 1} )$  и продолженные функции обозначим через $\tilde \Phi _k (x).$
Тогда в области $\tilde \Omega  = \{ (x,y):\,x \in R^n ,\,\,\,t > 0\} $  получим задачу нахождения решения уравнения (15), удовлетворяющее начальным условиям
\begin{equation} \label{ic18}
\left. {\frac{{\partial ^k U}}{{\partial t^k }}} \right|_{t = 0}  = \tilde \Phi _k (x), \,\, x \in R,\,\,\,k = \overline {0,m - 1},
\end{equation}

	Введем обозначения: $W_0 (x,t) = U(x,t)$ и $W_k (x,t) = \left( {\dfrac{\partial }{{\partial t}} - \Delta } \right)^k W_0 (x,t).$
 В этих обозначениях задача {(15), (18)} эквивалентна к следующей задаче  о нахождении функций $W_k (x,t),\,\,k = \overline {0,m - 1}, $ удовлетворяющих системе уравнений
\begin{equation} \label{eq19}
\left\{ \begin{gathered}
  \frac{{\partial W_k }}
{{\partial t}} - \Delta W_k  = W_{k + 1} ,\,\,(x,t) \in \tilde \Omega ,\,\,\,k = \overline {0,m - 2} , \hfill \\
  \frac{{\partial W_{m - 1} }}
{{\partial t}} - \Delta W_{m - 1}  = 0,\,\,\,(x,t) \in \tilde \Omega  \hfill \\
\end{gathered}  \right.
\end{equation}
и начальным условиям
\begin{equation} \label{ic20}
W_k (x,0) = F_k (x), \,\, x \in R^n ,\,\,\,\,\,k = \overline {0,m - 1},
\end{equation}
где
\begin{equation} \label{if21}
F_k (x) = \sum\limits_{j = 0}^k {( - 1)^{k - j} C_k^j \Delta^{k - j} \tilde \Phi _j (x)}, \,\,k = \overline {0,m - 1},
\end{equation}
$C_k^j  = k!/[j!(k - j)!]$ - биномиальные коэффициенты.

При решении задачи {(19), (20)} воспользуемся следующей леммой.
\begin{lemma} \label{lm1}
Если $g(x) \in L_1 (R^n ),$ то имеет место равенство
\[
\int\limits_0^t {\frac{{d\tau }}
{{\left( {2\sqrt {\pi (t - \tau )} } \right)^n }}\int\limits_{R^n } {\exp \left[ { - \frac{{\left| {x - y} \right|^2 }}
{{4(t - \tau )}}} \right]\left\{ {\frac{1}
{{\left( {2\sqrt {\pi \tau } } \right)^n }}\int\limits_{R^n } {g(\eta )\exp \left[ { - \frac{{(y - \eta )^2 }}
{{4\tau }}} \right]d\eta } } \right\}dy} }
\]
\begin{equation} \label{lm22}
 = \frac{t}
{{\left( {2\sqrt {\pi t} } \right)^n }}\int\limits_{R^n } {g(\eta )\exp \left[ { - \frac{{\left| {\eta  - x} \right|^2 }}
{{4t}}} \right]dy}.
\end{equation}
\end{lemma}

\begin{proof}  В левой части равенства (22), в силу равномерной сходимости несобственных интегралов, сделаем перестановку порядка интегрирования по $\eta $ и по $y.$ Затем, пользуясь формулой~\cite{pbm}
\[
\int\limits_{ - \infty }^{ + \infty } {\exp \left[ { - p\xi ^2  - q\xi } \right]d\xi }  = \sqrt {\frac{\pi }
{p}} \exp \left( {\frac{{q^2 }}
{{4p}}} \right),\,\,\,\,\,\operatorname{Re} p > 0,
\]
вычислив внутренний интеграл по $y,$ получим
\[
\prod\limits_{j = 1}^n {\int\limits_{ - \infty }^{ + \infty } {\exp \left[ { - \frac{{(x_j  - y_j )^2 }}
{{4(t - \tau )}} - \frac{{(y_j  - \eta _j )^2 }}
{{4\tau }}} \right]dy_j } }  =
\]
\begin{equation} \label{for23}
= \left[ {2\frac{{\sqrt \pi  }}{{\sqrt t }}\sqrt {\tau (t - \tau )} } \right]^n \exp \left[ { - \frac{{\left| {\eta  - x} \right|^2 }}{{4t}}} \right].
\end{equation}

Подставляя (23) в левую часть равенства (22), после сокращения подобных членов, получим справедливость утверждения леммы 1.
\end{proof}

Вернемся к исследованию задачи \{(19), (20)\}. Последовательно решая каждое уравнение системы (19) начиная с последнего, с учетом начальных условий (20) и леммы 1, находим решение задачи \{(19), (20)\}. Затем, учитывая $W_0 (x,t) = U(x,t),$ получим решение  задачи \{(15), (17)\} в виде
\begin{equation} \label{sp24}
U(x,t) = \left( {2\sqrt {\pi t} } \right)^{ - n} \sum\limits_{k = 0}^{m - 1} {\frac{{t^k }}
{{k!}}\int\limits_{R^n } {F_k (s)\exp \left[ { - \frac{{\left| {s - x} \right|^2 }}
{{4t}}} \right]ds} },
\end{equation}
где $F_k (x)$ ($k = \overline {0,m - 1} $) - известные функции, определяемые через заданные начальные функции равенствами (21).

	Учитывая четность функций $F_k (x),$ $k = \overline {0,m - 1}, $ равенство (24) перепишем в виде
\begin{equation} \label{sp25}
U(x,t) = \sum\limits_{k = 0}^{m - 1} {\frac{{t^k }}{{k!}}U_k (x,t)},
\end{equation}
здесь
\begin{equation} \label{sp26}
U_k (x,t) = \int\limits_{R_ + ^n } {F_k (s)G(x,t,s)ds},
\end{equation}
\[
G(x,s,t) = \prod\limits_{j = 1}^n {G_0 (x_j ,s_j ,t)},
\]
\[
G_0 (x_j ,s_j ,t) = \frac{1}
{{2\sqrt {\pi t} }}\left\{ {\exp \left[ { - \frac{{(s_j  - x_j )^2 }}
{{4t}}} \right] + \exp \left[ { - \frac{{(s_j  + x_j )^2 }}
{{4t}}} \right]} \right\}.
\]

	Чтобы исследовать поведение функций $F_k (x),$  $k = \overline {0,m - 1} $  сделаем некоторые преобразования. Для этого докажем следующую лемму.
\begin{lemma} \label{lm2}
Пусть функции $\varphi _j (x) \in C^{2(m - j) - 1} (R_ + ^n ),$ $j = \overline {0,m - 1} $  непрерывны,  ограничены и все производные начальных функций $\varphi _j (x),$ до порядка $2(m - j) - 1,$ $j = \overline {0,m - 1} $  включительно, обращаются в нуль при $x_k  = 0,\,\,k = \overline {1,n}. $
Тогда имеют место равенства
\begin{equation} \label{lm27}
\mathop {\lim }\limits_{x_k  \to 0} x_k^{2\alpha _k } \frac{\partial }{{\partial x_k }}\left[ {B_{\alpha _k  - (1/2)}^{x_k } } \right]^l \varphi _j (x) = 0, \,\, k = \overline {1,\,n}, \,\, l = \overline {0,m - 1}, \,\, j = \overline {0,m - 1},
\end{equation}
\begin{equation} \label{lm28}
\mathop {\lim }\limits_{x_k  \to 0} [B_{\gamma _k }^{x_k } ]^i \varphi _{j - i} (x) = 0, \,\, k = \overline {1,\,n}, \,\, i = \overline {0,j}, \,\, j = \overline {0,m - 1}.
\end{equation}
\end{lemma}

\begin{proof}  Методом математической индукции нетрудно убедиться в справедливости следующего равенства
\begin{equation} \label{lm29}
\left( {\frac{1}{x}\frac{d}{{dx}}} \right)^p h(x) = \sum\limits_{j = 1}^p {( - 1)^{j + 1} A_{pj} \frac{{h^{(p - j + 1)} (x)}}{{x^{p + j - 1} }}},
\end{equation}
где  $A_{pj} $ - постоянные, определяемые из следующих рекуррентных равенств
\[
A_{(p + 1)1}  = A_{p1}  = 1,\,\,p \geqslant 1, \,\, A_{(p + 1)j}  = (p + j - 1)A_{p(j - 1)}  + A_{pj}, \,\,
p \geqslant 2,\,\,j = \overline {2,p},
\]
\[
 A_{(p + 1)(p + 1)}  = (2p - 1)A_{pp}  = (2p - 1)!!,\,\,\,p \geqslant 1.
\]

Равенство (27) перепишем в виде
\[
\mathop {\lim }\limits_{x_k  \to 0} H(x) = \mathop {\lim }\limits_{x_k  \to 0} x_k^{2\alpha _k } \frac{\partial }
{{\partial x_k }}\left[ {B_{\alpha _k  - (1/2)}^{x_k } } \right]^l \varphi _j (x) =
\]
\[
 = \mathop {\lim }\limits_{x_k  \to 0} x_k^{1 + 2\alpha _k } \sum\limits_{q = 0}^l {C_l^q (2\alpha _k )^{l - q} \left( {\frac{1}
{{x_k }}\frac{\partial }
{{\partial x_k }}} \right)^{l - q + 1} \varphi _j^{(2q)} (x)}.
\]

Учитывая (29), имеем
\[
\mathop {\lim }\limits_{x_k  \to 0} H(x) = \sum\limits_{q = 0}^l {C_l^q (2\alpha _k )^{l - q} } \sum\limits_{j = 1}^{l - q + 1} {( - 1)^{j + 1} A_{(l - q + 1)j} \mathop {\lim }\limits_{x_k  \to 0} \frac{{\varphi _j^{(l - q - j + 2)} (x)}}{{x_k^{l - q + j - 2 - 2\alpha _k } }}}.
\]

Применяя к последнему равенству правило Лопиталя $l - q + j - 2$  раз  \cite{iss} и учитывая условие доказываемой леммы, получим
\[	
\mathop {\lim }\limits_{x_k  \to 0} \frac{{\varphi _j^{(l - q - j + 2)} (x)}}
{{x_k^{l - q + j - 2 - 2\alpha _k } }} = \frac{{\mathop {\lim }\limits_{x_k  \to 0} x_k^{2\alpha _k } \varphi _j^{(2(l - q))} (x)}}
{{(l - q + j - 2)!}} = 0
\]
	
	Отсюда следует справедливость равенства (27). Равенство (28) доказываются аналогично. Лемма 2 доказано.
\end{proof}

В силу леммы 2 для функций $\Phi _k (x)$ выполняются условия теоремы 7. Поэтому, принимая во внимание формулу (13), равенство (21) при $x_k  > 0,\,\,k = \overline {1,n} $
 можно представить в виде
\begin{equation} \label{if30}
F_k (x) = J_0^{ - 1} \left( \begin{gathered}
  \,\,\,\,\,\,\alpha  \hfill \\
   - 1/2 \hfill \\
\end{gathered}  \right)f_k (x), \,\, k = \overline {0,m - 1},
\end{equation}
где
\begin{equation} \label{if31}
f_k (x) = \sum\limits_{j = 0}^k {( - 1)^j C_k^j \Delta _B^j \varphi _{k - j} (x)}, \,\, k = \overline {0,m - 1}.
\end{equation}

Учитывая вид обратного оператора (7), равенства (30) представим в виде  $F_k (x) = [\partial ^n /(\partial x_1 \partial x_2 ...\partial x_n )]\bar F_k (x),$ где $k = \overline {0,m - 1},$
\[
\bar F_k (x) = \prod\limits_{j = 1}^n {\left[ {\frac{1}
{{\Gamma (1 - \alpha _j )}}} \right]\int\limits_0^{x_1 } {\int\limits_0^{x_2 } {...\int\limits_0^{x_n } {\prod\limits_{j = 1}^n {\left[ {(x_j^2  - s_j^2 )^{ - \alpha _j } s_j^{2\alpha _j } } \right]} } } f_k (s)ds_1 ds_2 ...ds_n } }.
\]

	Заметим, что в силу леммы 2  из равенства (31) следует, что  функции $f_k (x),$ $k = \overline {0,m - 1} $
 при $x_j  \geqslant 0$ будут  непрерывными, ограниченными и $\left. {f_k (x)} \right|_{x_j  = 0}  = 0,$ в силу чего из последнего равенства, имеем
\begin{equation} \label{if32}
\left. {\bar F_k (x)} \right|_{x_j  = 0}  = 0,\,\,\,j = \overline {1,n}, \,\, k = \overline {0,m - 1}.
\end{equation}

	Принимая во внимание (32), в равенстве (26) выполним интегрирования по частям. Затем, подставив в это равенство значение функций $\bar F_k (x),$ получим
\begin{equation} \label{sp33}
U_k (x,t) =  - \prod\limits_{j = 1}^n {\left[ {\frac{1}
{{\Gamma (1 - \alpha _j )}}} \right]} \int\limits_{R_ + ^n } {f_k (s)\prod\limits_{j = 1}^n {\left[ {s_j^{2\alpha _j } G_1 (x_j ,s_j ,t)} \right]} \,ds},
\end{equation}
где
\begin{equation} \label{gf34}
G_1 (x_j ,s_j ,t) = \int\limits_{s_j }^{ + \infty } {(y_j^2  - s_j^2 )^{ - \alpha _j } \frac{\partial }
{{\partial y_j }}G_0 (x_j ,y_j ,t)dy_j }.
\end{equation}

Вычислим интеграл (34). Применяя  формулу (\cite{pbm}, с. 451)
\[
\int\limits_0^{ + \infty } {e^{ - a\lambda ^2 } \cos (b\lambda )d\lambda }  = \sqrt {\frac{\pi }
{{4a}}} \exp \left[ { - \frac{{b^2 }}
{{4a}}} \right],\,\,\,\,\operatorname{Re} a > 0,
\]
функцию   $G_0 (x_j ,y_j ,t)$  представим в виде
\[
G_0 (x_j ,y_j ,t) = \frac{2}
{\pi }\int\limits_0^{ + \infty } {e^{ - t\lambda ^2 } \cos (x_j \lambda )\cos (y_j \lambda )d\lambda }.
\]

Отсюда, вычислим производную по $y_j $  и полученное выражение для функции $G_{0y} $  подставим в (34). Затем, принимая во внимание равномерную сходимость интегралов, меняем порядок интегрирования и, применяя формулу Мелера-Сонина (\cite{be}, с. 93), вычислим внутренний интеграл. В результате находим
\[
G_1 (x_j ,s_j ,t) =
\]
\begin{equation} \label{gf35}
=  - \frac{{2^{(1/2) - \alpha _j } }}
{{\sqrt \pi  }}\Gamma (1 - \alpha _j )s_j^{(1/2) - \alpha } \int\limits_0^{ + \infty } {e^{ - t\lambda ^2 } \lambda ^{\alpha _j  + (1/2)} J_{\alpha _j  - (1/2)} (\lambda s_j )\cos (x_j \lambda )d\lambda},
\end{equation}
где  $J_v (z)$- функция Бесселя первого рода порядка $v$ (\cite{be}, с. 12).

Теперь, подставляя (33) в (25), а его в (14), после смены порядка интегрирования, получим
\begin{equation} \label{sp36}
u(x,t) =  - \prod\limits_{j = 1}^n {\left[ {\frac{{2x_j^{1 - 2\alpha _j } }}
{{\Gamma (\alpha _j )\Gamma (1 - \alpha _j )}}} \right]} \sum\limits_{k = 0}^{m - 1} {\,\frac{{t^k }}
{{k!}}\int\limits_{R_ + ^n } {f_k (s)\prod\limits_{j = 1}^n {s_j^{2\alpha _j } G_2 (x_j ,s_j ,t)} ds} }
\end{equation}
где
\begin{equation} \label{gf37}
G_2 (x_j ,s_j ,t) = \int\limits_0^{x_j } {(x_j^2  - \xi _j^2 )^{\alpha _j  - 1} \;G_1 (\xi _j ,s_j ,t)d\xi _j }
\end{equation}

Подставим в (37) выражение (35) функции $G_1$  и меняем порядок интегрирования. Затем, применяя формулу Пуассона (\cite{be}, с. 93),  вычислим внутренний интеграл. В итоге находим
\[
G_2 (x_j ,s_j ,t) =
\]
\[
= - \frac{1}
{2}\Gamma (\alpha _j )\Gamma (1 - \alpha _j )\left( {\frac{{s_j }}
{{x_j }}} \right)^{(1/2) - \alpha _j } \int\limits_0^\infty  {e^{ - t\lambda ^2 } J_{\alpha _j  - (1/2)} (s_j \lambda )J_{\alpha _j  - (1/2)} (x_j \lambda )\lambda d\lambda }.
\]

Далее, учитывая следующую формулу (\cite{be}, с. 60)
\[
\int\limits_0^\infty  {e^{ - t\lambda ^2 } J_\nu  (s\lambda )J_\nu  (x\lambda )\lambda d\lambda }  = \frac{1}
{{2t}}e^{ - \frac{{x^2  + s^2 }}
{{4t}}} I_\nu  \left( {\frac{{xs}}
{{2t}}} \right),
\]
$\operatorname{Re} \nu  >  - 1,\,\,\operatorname{Re} t > 0,$ имеем
\begin{equation} \label{gf38}
G_2 (x_j ,s_j ,t) =  - \frac{1}
{{4t}}\Gamma (\alpha _j )\Gamma (1 - \alpha _j )\left( {\frac{{s_j }}
{{x_j }}} \right)^{(1/2) - \alpha _j } e^{ - \frac{{x_j^2  + s_j^2 }}
{{4t}}} I_{\alpha _j  - (1/2)} \left( {\frac{{x_j s_j }}{{2t}}} \right),
\end{equation}
где  $I_v (z)$  - функция  Бесселя мнимого аргумента порядка $v$  (\cite{be}, с. 13).

Подставляя (38) в (36) и учитывая $\alpha _j  = \gamma _j  + (1/2) < 1$  и $\gamma _j  >  - (1/2),$ $j = \overline {1,n},$ находим окончательный вид решения однородного уравнения  $L_\gamma ^m (u) = 0,$
удовлетворяющее условиям (2) и (3) при $\left| {\gamma _j } \right| < 1/2,$ $j = \overline {1,n}:$
\begin{equation} \label{sp39}
u(x,t) = \frac{1}{{(2t)^n }}\prod\limits_{j = 1}^n {x_j^{ - \gamma _j } } \sum\limits_{k = 0}^{m - 1} {\frac{{t^k }}{{k!}}\,\int\limits_{R_ + ^n } {f_k (s)G(x,s,t)ds} },
\end{equation}
где $f_k (x) = \sum\limits_{j = 0}^k {( - 1)^j C_k^j \Delta _B^{k - j} \varphi _j (x)}, $
\[
G(x,s,t) = \prod\limits_{j = 1}^n {\left\{ {s_j^{\gamma _j  + 1} \exp \left[ { - \frac{{x_j^2  + s_j^2 }}
{{4t}}} \right]I_{\gamma _j } \left( {\frac{{x_j s_j }}
{{2t}}} \right)} \right\}}  =
\]
\begin{equation} \label{gf40}
 = \prod\limits_{j = 1}^n {\left[ {s_j^{\gamma _j  + 1} I_{\gamma _j } \left( {\frac{{x_j s_j }}
{{2t}}} \right)} \right]} \exp \left[ { - \frac{{\left| x \right|^2  + \left| s \right|^2 }}
{{4t}}} \right], \,\, \left| x \right|^2  = \sum\limits_{j = 1}^n {x_j^2 }.
\end{equation}

	Непосредственной проверкой можно убедиться, что справедлива следующая теорема.
\begin{theorem} \label{th8}
Пусть  $\left| {\gamma _j } \right| < 1/2,$ $j = \overline {1,n}, $ а функции $\varphi _j (x) \in C^{2(m - j) - 1} (R_ + ^n ),$ $j = \overline {0,m - 1} $  непрерывны, ограничены и все производные функций $\varphi _j (x),$
до порядка $2(m - j) - 1,$ $j = \overline {0,m - 1} $  включительно, обращаются в нуль при $x_k  = 0,\,\,k = \overline {1,n}. $
Тогда функция $u(x,t),$ определяемая равенством (39), является классическим решением однородного уравнения $L_\gamma ^m (u) = 0,$ удовлетворяющее условиям (2) и (3).
\end{theorem}

\section{Решение поставленной задачи в случае неоднородного уравнения}

Теперь рассмотрим задачу нахождения решения неоднородного уравнения (1) удовлетворяющего однородным граничным условиям (3) и следующим однородным начальным условиям
\begin{equation} \label{ic41}
\left. {\frac{{\partial ^k u}}{{\partial t^k }}} \right|_{t = 0}  = 0, \,\, x \in R_ + ^n, \,\,
k = \overline {0,m - 1}.
\end{equation}

	Чтобы найти решение задачи \{(1), (3), (41)\} воспользуемся следующим аналогом второго принципа Дюамеля для уравнения высокого порядка.

\begin{theorem} \label{th9}
Пусть функция $U(x,t,\tau ),$ зависящая от параметра $\tau, $ является решением однородного уравнения $L_\gamma ^m (U) = 0,\,\,x \in R_ + ^n ,\,\,\,t > \tau, $ удовлетворяющего однородным граничным условиям (3) при $t > \tau $  и следующим начальным условиям
\begin{equation} \label{ic42}
\left. {\frac{{\partial ^k U}}{{\partial t^k }}} \right|_{t = \tau }  = 0,\,x \in R_ + ^n ,\,\,\,\,k = \overline {0,m - 2}, \,\, \left. {\frac{{\partial ^{m - 1} U}}{{\partial t^{m - 1} }}} \right|_{t = \tau }  = f(x,\tau ),\,x \in R_ + ^n.
\end{equation}
	Тогда функция
\begin{equation} \label{sp43}
V(x,t) = \int\limits_0^t {U(x,t,\tau )d\tau },
\end{equation}
будет решением задачи \{(1), (3), (41)\}.
\end{theorem}

\begin{proof}  Дифференцируя равенство (43) и учитывая (42), находим
\[
\frac{{\partial V}}
{{\partial t}} = \left. {U(x,t,\tau )} \right|_{\tau  = t}  + \int\limits_0^t {\frac{\partial }
{{\partial t}}U(x,t,\tau )d\tau }  = \int\limits_0^t {\frac{\partial }
{{\partial t}}U(x,t,\tau )d\tau },
\]
\[
\Delta _B V(x,t) = \int\limits_0^t {\Delta _B U(x,t,\tau )d\tau }.
\]

	Составим выражение
\[
L_\gamma ^1 (V) = \left( {\frac{\partial }
{{\partial t}} - \Delta _B } \right)V = \int\limits_0^t {L_\gamma ^1 \left( {U(x,t,\tau )} \right)d\tau }.
\]

	Повторяя этот процесс $m - 1$  раз, с учетом (42) получим
\begin{equation} \label{sp44}
L_\gamma ^{m - 1} (V) = \left( {\frac{\partial }
{{\partial t}} - \Delta _B } \right)^{m - 1} V = \int\limits_0^t {L_\gamma ^{m - 1} \left( {U(x,t,\tau )} \right)d\tau }.
\end{equation}

	Далее, дифференцируя (44) по $t,$ находим
\begin{equation} \label{sp45}
\frac{\partial }{{\partial t}}L_\gamma ^{m - 1} (V) = \left. {L_\gamma ^{m - 1} \left( {U(x,t,\tau )} \right)} \right|_{\tau  = t}  + \int\limits_0^t {\frac{\partial }{{\partial t}}L_\gamma ^{m - 1} \left( {U(x,t,\tau )} \right)d\tau }.
\end{equation}

	В силу (42), имеем
\[
\left. {L_\gamma ^{m - 1} \left( {U(x,t,\tau )} \right)} \right|_{\tau  = t}  = \left( {\frac{\partial }
{{\partial t}} - \Delta _B } \right)^{m - 1} \left. U \right|_{\tau  = t}  =
\]
\[
 = \sum\limits_{k = 0}^{m - 1} {C_{m - 1}^k [ - \Delta _B ]^{m - k - 1} \left. {\frac{{\partial ^k U}}
{{\partial t^k }}} \right|_{\tau  = t}  = } \left. {\frac{{\partial ^{m - 1} U}}
{{\partial t^{m - 1} }}} \right|_{\tau  = t}  = f(x,t).
\]

	Принимая во внимание последнее равенство, из (45) получим
\begin{equation} \label{sp46}
\frac{\partial }{{\partial t}}L_\gamma ^{m - 1} (V) = f(x,t) + \int\limits_0^t {\frac{\partial }
{{\partial t}}L_\gamma ^{m - 1} \left( {U(x,t,\tau )} \right)d\tau }.
\end{equation}

	Применяя к равенству (44) оператор $\Delta _B, $ а затем, вычитая полученное равенство из (46), в силу $L_\gamma ^m (U) = 0,$ находим $L_\gamma ^m (V) = f(x,t).$

	Непосредственным вычислением можно проверить, что функция $ V(x,t), $  определяемая равенством (43), удовлетворяет граничным условиям (3) и начальным условиям (41).  Теорема 9 доказана.
\end{proof}

	Чтобы найти решение однородного уравнения $L_\gamma ^m (U) = 0,\,\,x \in R_ + ^n ,\,\,\,t > \tau, $
удовлетворяющего граничным условиям (3) и начальным условиям (41) сделаем замену переменных $t_1  = t - \tau. $
Тогда решение этой задачи даётся формулой (39), в которой $t$  заменено на $t_1  + \tau. $
Из этой формулы возвращаясь к старым переменным $t$  и $\tau, $ имеем
\[
U(x,t,\tau ) = \frac{{(t - \tau )^{m - n - 1} }}
{{2^n (m - 1)!}}\prod\limits_{j = 1}^n {\left[ {x_j^{ - \gamma _j } } \right]} \,\int\limits_{R_ + ^n } {f(s,\tau )G(x,s,t - \tau )ds}.
\]

	Принимая во внимание последнее равенство, из (43) находим решение задачи \{(1), (3), (41)\} в следующем виде
\[
V(x,t) = \frac{1}
{{2^n (m - 1)!}}\prod\limits_{j = 1}^n {\left[ {x_j^{ - \gamma _j } } \right]} \,\int\limits_0^t {(t - \tau )^{m - n - 1} d\tau } \int\limits_{R_ + ^n } {f(s,\tau )G(x,s,t - \tau )ds},
\]
где  $G(x,s,t)$  - функция, определяемая равенством (40).

\newpage
%Обязательна информация об авторе
\begin{center}
{\bf ANALOG OF THE CAUCHY PROBLEM FOR THE INHOMOGENEOUS MANYDIMENSIONAL POLYCALORIC EQUATION WITH BESSEL OPERATOR}
\end{center}

Karimov Shakhobiddin Tuychiboyevich ---
Candidate of Physical and Mathematical Sciences, Docent,
Chair of the Department of Mathematics, Ferghana State
University. Ferghana, Murabbiylar str. 19, FerSU, 150100, Republic of Uzbekistan. \\
E-mail: shkarimov09@rambler.ru \\

MSC: 35K25, 35K30.

\begin{abstract}
In the work an explicit formula of a solution of analogue of a Cauchy problem for an inhomogeneous manydimensional polycaloric equation with Bessel operator was found. Manydimensional Erd\'{e}lyi-Kober operator  of the fractional order was applied to construction of a solution.
	
Key words:  Cauchy problem, the polycaloric equation, Erd\'{e}lyi-Kober operator, Bessel differential operator.
\end{abstract}
\bigskip

\end{document}